\documentclass[psamsfonts]{amsart}

\usepackage{hyperref}
\hypersetup{
    colorlinks,
    citecolor=black,
    filecolor=black,
    linkcolor=black,
    urlcolor=black
}
\usepackage{geometry}
\usepackage{pxfonts}
\usepackage[fontsize=11pt]{scrextend}
\usepackage{amssymb,amsfonts}
\usepackage[all,arc]{xy}
\usepackage{enumerate}
\usepackage{mathrsfs}
\usepackage[utf8]{inputenc}
\usepackage{url}
\usepackage{mathtools}
\usepackage{enumerate}
\usepackage[shortlabels]{enumitem} 
\usepackage{tikz}
\usetikzlibrary{trees, snakes}
\usepackage{color}
\usepackage{bbm}
\usepackage{graphicx}
\usepackage{enumitem}

\calclayout

\DeclareMathOperator\Fr{Fr}
\DeclareMathOperator\Gr{Gr}
\DeclareMathOperator\PSL{PSL}

\DeclareMathOperator\supp{supp}

\DeclareMathOperator\Span{span}

\DeclareMathOperator\SO{SO}

\DeclareMathOperator\Isom{Isom}

\DeclareMathOperator\ft{\mathbf{f}\bt}

\DeclareMathOperator\core{core}

\DeclareMathOperator\N{N}

\DeclareMathOperator\acc{acc}
\DeclareMathOperator\UT{UT}

\newcommand{\lef}{\left}
\newcommand{\ri}{\right}

\newcommand{\wk}{\rightharpoonup}
\newcommand{\wkstar}{\stackrel{\star}\wk}

\newcommand{\aand}{\quad\text{and}\quad}
\newcommand{\eps}{\epsilon}

\newcommand{\se}{\subset}

\newcommand{\pants}{\mathbf{\Pi}_{\epsilon, R}}

\newcommand{\curves}{\mathbf{\Gamma}_{\epsilon, R}}

\newcommand{\bC}{\mathbf{C}}

\newcommand{\bH}{\mathbf{H}}

\newcommand{\bQ}{\mathbf{Q}}
\newcommand{\bR}{\mathbf{R}}

\newcommand{\bZ}{\mathbf{Z}}

\newcommand{\bt}{\mathbf{t}}

\makeatletter
\def\section{\@startsection{section}{1}%
\z@{.7\linespacing\@plus\linespacing}{1.0\linespacing}%
{\normalfont\centering\Large}}
\def\subsection{\@startsection{subsection}{2}%
\z@{.5\linespacing\@plus.7\linespacing}{-.5em}%
{\normalfont\bfseries}}
\def\subsubsection{\@startsection{subsubsection}{3}%
\z@{.5\linespacing\@plus.7\linespacing}{-.5em}%
{\normalfont\bfseries}}
\makeatother

\makeatletter
\def\@settitle{\begin{center}%
  \baselineskip14\p@\relax
    \normalfont\LARGE%<- NEW
  \@title
  \end{center}%
}
\makeatother

\newtheorem{thm}{Theorem}[section]
\newtheorem{cor}[thm]{Corollary}
\newtheorem{prop}[thm]{Proposition}
\newtheorem{lem}[thm]{Lemma}

\newtheorem{claim}[thm]{Claim}
\newtheorem*{claim*}{Claim}
\newtheorem*{thm*}{Theorem}
\newtheorem*{appl*}{Application}
\newtheorem*{prop*}{Proposition}
\newtheorem*{lem*}{Lemma}
\newtheorem*{cor*}{Corollary}
\newtheorem*{conj*}{Conjecture}

\theoremstyle{definition}
\newtheorem{ques}{Question}

\newtheorem*{ques*}{Question}
\newtheorem*{exmp*}{Example}
\newtheorem*{defn*}{Definition}

\def\XXint#1#2#3{{\setbox0=\hbox{$#1{#2#3}{\int}$ }
\vcenter{\hbox{$#2#3$ }}\kern-.6\wd0}}

\theoremstyle{remark}
\newtheorem{rem}[thm]{Remark}
\newtheorem*{rem*}{Remark}

\makeatletter
\let\c@equation\c@thm
\makeatother
\numberwithin{thm}{section}
\numberwithin{equation}{section}

\title{ Asymptotically geodesic surfaces  }
\author{Fernando Al Assal and Ben Lowe}

\begin{document}

\maketitle

\begin{abstract}

A sequence of distinct  closed surfaces in a hyperbolic 3-manifold $M$ is \emph{asymptotically geodesic} if their principal curvatures tend uniformly to zero. When $M$ has finite volume, we show such sequences are always asymptotically dense in the 2-plane Grassmann bundle of $M$. When $M$ has infinite volume and is geometrically finite, we show such sequences do not exist. As an application of the former, we obtain partial answers to the question of whether a negatively curved Riemannian 3-manifold that contains a sequence of asymptotically totally geodesic or totally umbilic surfaces must be hyperbolic.  Finally, we give examples to show that if the dimension of $M$ is greater than 3, the possible limiting behavior of asymptotically geodesic surfaces is less constrained than for totally geodesic surfaces.

\end{abstract}

\section{Introduction}

A sequence of distinct connected closed surfaces in a hyperbolic 3-manifold $M$ is \emph{asymptotically geodesic} if their principal curvatures tend to zero in $L^{\infty}$ norm.   Kahn-Markovi\'{c} and Kahn-Wright showed that when $M$ has finite volume, such sequences of surfaces exist abundantly \cite{kahn2012immersing}, \cite{KW}. The first theorem (Theorem \ref{contra}) in this article implies that such sequences $(S_n)$ of surfaces are always \emph{asymptotically dense} in the 2-plane Grassmann bundle $\Gr M$ of $M$ -- given an open $B\se \Gr M$, there is $N$ so for every $n\geq N$, $S_n$ intersects $B$. This result is used to obtain partial answers to variations on the question of whether a closed negatively curved Riemannian manifold that contains a sequence of asymptotically geodesic surfaces is hyperbolic (Theorems \ref{asymprigidity}, \ref{intro:totallyumbilic}.) 

We also show (Theorem \ref{intro:acylindricalgap}) that sequences of distinct asymptotically geodesic surfaces do not exist in geometrically finite hyperbolic 3-manifolds $M$ of \emph{infinite volume}. Precisely, we show there is $\eps(M)>0$ so that a closed surface of principal curvatures smaller than $\eps$ in absolute value is homotopic to a totally geodesic surface.

There can be real qualitative differences in the behavior of sequences of totally geodesic surfaces and sequences of surfaces that are only asymptotically  geodesic.  For example, the first author showed that a sequence of distinct asymptotically geodesic surfaces $\Sigma_n$ in a hyperbolic 3-manifold can scar along a closed totally geodesic surface $\Sigma$, in the sense that the corresponding sequence of probability measures $\mu_{\Sigma_n}$ converges to the probability measure corresponding to $\Sigma$ \cite{a}.  This is in contrast to what happens for sequences of distinct closed totally geodesic $\Sigma_n$ in $M$, which Mozes-Shah showed must become uniformly distributed \cite{MS}. 

Most of our theorems are in the opposite direction, and show that asymptotically geodesic surfaces in negative curvature enjoy many of the same rigidity properties as totally geodesic surfaces. In fact, many of our proofs rely on theorems from homogeneous dynamics that establish those rigidity properties for totally geodesic surfaces, such as the fact proven by Ratner and Shah that they are either closed or dense in a hyperbolic 3-manifold of finite volume \cite{R},\cite{S}.

On the other hand, we construct sequences of asymptotically Fuchsian ($K_n$-quasifuchsian for $K_n\to 1$) pleated surfaces and minimal surfaces in certain hyperbolic $d$-manifolds $M$ with $d\geq 4$ whose Hausdorff limit is a union of two distinct totally geodesic submanifold of $M$ (Theorem \ref{higherintro}.) This contrasts with the fact, that follows from the work of Mozes-Shah, that the Hausdorff limit of a sequence of totally geodesic surfaces in $M$, if it exists, has to be a single totally geodesic $k$-submanifold for $3 \leq k \leq d$ (Proposition \ref{prop:totallygeodesicaccumulation}.) 

\subsection{Density of asymptotically geodesic surfaces in finite volume} Let $M = \Gamma\backslash \bH^3$ be a hyperbolic 3-manifold of finite volume, where $\Gamma\leq \PSL_2\bC$ is a lattice. All surfaces in this paper will be connected.  An essential surface $S\se M$ is \emph{$K$-quasifuchsian} if $\pi_1(S) \leq \pi_1 M = \Gamma$ is a $K$-quasifuchsian subgroup of $\PSL_2\bC$, i.e., there is a $K$-quasiconformal homeomorphism $\phi: \partial_{\infty} \bH^3 \to \partial_{\infty} \bH^3$ so $\phi \circ \pi_1(S) \circ \phi^{-1}$ is Fuchsian. For $K$ sufficiently close to 1, work by Uhlenbeck and Seppi imply that $K$-quasifuchsian surfaces $S\se M$ are homotopic to unique minimal surfaces with principal curvatures going to zero uniformly as $K\to 1$ \cite{U},\cite{Se}.  All essentials immersed surfaces in this paper are maximal: they are not homotopic to covers of some surface in $M$ of smaller genus.  Our first theorem is: 

\begin{thm}\label{contra} Let $M$ be a finite volume hyperbolic 3-manifold. Let $B$ be an open set. Then, there is $\eps=\eps(B)>0$ such that all but finitely many of the $(1+\eps)$-quasifuchsian minimal surfaces of $M$ meet $B$. Moreover, the ones which do not are totally geodesic.
\end{thm}

As direct consequences of the previous theorem (and its proof), we obtain the following two corollaries.  We say that a surface $S$ is $\epsilon$\textit{-dense} in $\Gr M$ if every tangent 2-plane to $M$ is at a distance of at most $\epsilon $ from some tangent plane to $S$.  

\begin{cor} \label{epsilondensecorollary}
For every $\epsilon>0$ there is some $\eta = \eta(M)>0$ so that all but finitely many  $(1+\eta)$-quasifuchsian closed minimal surfaces are $\eps$-dense in the $\eta$-thick part of $M$.  
\end{cor}

\begin{cor}\label{dense}
Let $(S_n)$ be a sequence of distinct closed  asymptotically Fuchsian minimal surfaces in $M$. Then, $(S_n)$ is asymptotically dense in $\Gr M$: for every open $U\se \Gr M$, there is $N$ so that for every $n\geq N$, $S_n$ intersects $U$.
\end{cor}

These theorems imply that, in the topology induced by the Hausdorff distance between closed sets in $\Gr M$, sequences of asymptotically geodesic surfaces $S_n$ always limit to all of $\Gr M$. This is in contrast with the fact that sequences of asymptotically geodesic surfaces $S_n$ may not equidistribute in $\Gr M$ -- in fact, work by the first author showed that the probability measures $\nu_n$ induced by them on $\Gr M$ may limit in the weak-* topology to a measure supported on a single totally geodesic surface in $M$ \cite{a}. 

A crucial ingredient in these proofs is the Ratner-Shah theorem stating that totally geodesic surfaces are either closed or dense in $M$. When $M$ has infinitely many distinct totally geodesic surfaces, the fact due to Mozes-Shah that they are equidistributing also plays an important role \cite{MS}.

\begin{rem}
Results similar to the previous theorem and corollaries were recently and independently obtained by Xiaolong Hans Han \cite{hans}.  
\end{rem}

\subsection{Hausdorff limits in higher dimensions}

We say a sequence of surfaces $(S_n)$ in a hyperbolic $d$-manifold $M$ is \emph{asymptotically Fuchsian} if the $S_n$ are $K_n$-quasifuchsian for $K_n\to 1$ as $n\to \infty$. (Analogously to the 3-dimensional case, we say an essential surface $S\se M$ is $K$-quasifuchsian if $\pi_1(S)\leq \SO(d,1)$ is a $K$-quasifuchsian subgroup, meaning there is a $K$-quasiconformal homeomorphism of $\partial_{\infty}\bH^d$ conjugating $\pi_1(S)$ into $\SO(2,1)$.)

Asymptotically Fuchsian surfaces in $M$ are always homotopic to asymptotically geodesic ones \cite{j}. Conversely, asymptotically geodesic surfaces are asymptotically Fuchsian \cite{epstein1986hyperbolic}. We show

\begin{thm}\label{higherintro}
Suppose $M$ is a closed hyperbolic $d$-manifold containing two closed totally geodesic 3-dimensional submanifolds $N_1$ and $N_2$ that intersect along a closed totally geodesic surface.

Then, there exists a sequence $(S_n)$ of asymptotically Fuchsian closed connected pleated surfaces whose Hausdorff limit in $M$ is $N_1\cup N_2$.  There also exists a sequence $(S_n)$ of asymptotically geodesic minimal surfaces whose Hausdorff limit in $M$ is $N_1\cup N_2$.
\end{thm}

This is in contrast with the fact that the Hausdorff limit of a sequence of connected closed \emph{totally geodesic} surfaces in a closed hyperbolic $d$-manifold is always a single totally geodesic submanifold when it exists (see Section \ref{section:pants}.)

To construct the surfaces $(S_n)$, we construct sequences $(S^1_n)$ and $(S^2_n)$ of closed, connected, equidistributing surfaces on each of the $N_i$. These surfaces are built out of \emph{good pants} -- pants with large but nearly identical cuff lengths, following Kahn-Markovi\'{c} and Kahn-Wright, as well as Liu-Markovi\'{c} to ensure they are connected. We join the sequences inside each $N_i$ with a surgery and argue they remain asymptotically Fuchsian, using the fact that the totally geodesic surface $N_1\cap N_2$ has a finite cover built out of the same building blocks as $(S^1_n)$ and $(S^2_n)$ (as shown by Kahn-Markovi\'{c} in their proof of the Ehrenpreis conjecture \cite{KME}).

\subsection{Nonexistence of asymptotically geodesic surfaces in infinite volume}

We now let $M$ be a geometrically finite hyperbolic 3-manifold of \emph{infinite volume}. We say a surface $S\se M$ is \emph{essential} if it is $\pi_1$-injective.  

\begin{thm} \label{intro:acylindricalgap}
Let $M$ be a geometrically finite hyperbolic 3-manifold of infinite volume. Then, there is $\eps=\eps(M)>0$ so that any closed essential surface $S\se M$ with principal curvatures at most $\eps$ in absolute value is homotopic to a totally geodesic surface.

Similarly, let $\Gamma < \PSL_2\bC$ be a geometrically finite Kleinian group of infinite covolume. There is $\eps=\eps(\Gamma)>0$ so that any $(1+\eps)$-quasifuchsian surface subgroup of $\Gamma$ is Fuchsian.
\end{thm}

Combined with the fact proven by McMullen-Mohammadi-Oh \cite{mcmullen2017geodesic} and Benoist-Oh \cite{BO} that such $M$ only contain finitely many compact totally geodesic sufaces, this shows that $M$ contains no sequences of distinct closed asymptotically geodesic surfaces. This in contrast with the case when $M$ has finite volume, where the theorems of Kahn-Markovi\'{c} and Kahn-Wright show that there are many sequences of closed asymptotically geodesic surfaces.

On the other hand, when $M$ is also acylindrical, Cooper, Long and Reid show $M$ contains infinitely many maximal essential homotopy classes of surfaces that cannot be homotoped to a connected component of the boundary of the convex core \cite{cooper1997essential}.  Theorem \ref{intro:acylindricalgap} shows they are not $(1+\eps(\Gamma))$-quasifuchsian, except for possibly finitely many that are totally geodesic.

When the convex core of $M$ is compact and has totally geodesic boundary, we are able to prove Theorem \ref{intro:acylindricalgap} by contradiction -- if $M$ had a sequence of closed asymptotically geodesic surfaces, their minimal representatives would all live inside the convex core of $M$, so they would not be dense in the \emph{double} of $M$, violating Corollary \ref{dense}. In the general case, we also give an argument by contradiction, but using the theorem of Benoist-Oh (following McMullen-Mohammadi-Oh) that there are only finitely many compact totally geodesic surfaces in the convex core of $M$.

For $M$ as above, let $\eps(M)$ be the supremum over all $\eps>0$ such that $(1+\eps)$-quasifuchsian essential surfaces of $M$ are Fuchsian. We have shown $\eps(M)>0$, but we can still ask

\begin{ques} \footnote{This question was suggested by Alan Reid.}
Is it possible to estimate $\eps(M)$ in terms of geometric or topological properties $M$? Is the infimum of $\eps(M)$ over the deformation space of $M$ positive?
\end{ques}

As discussed above, finite volume hyperbolic 3-manifolds contain a great many sequences of closed non-totally geodesic hypersurfaces with principal curvatures tending to 0.  It is unknown whether finite volume higher dimensional hyperbolic manifolds are more like infinite volume hyperbolic 3-manifolds or finite volume hyperbolic 3-manifolds in this regard.  

\begin{ques}
Let $M$ be a finite volume hyperbolic $d$-manifold of dimension $d\geq 4$. Given $\eps>0$, does $M$ contain a closed essential hypersurface with principal curvatures at most $\eps$ in absolute value that is not homotopic to a totally geodesic hypersurface? Is there any such $M$ that does? 
\end{ques}

\subsection{Rigidity of asymptotically geodesic surfaces in negative curvature}

We are interested in variations on the following question. Corollary \ref{dense} will play an important role in the results we obtain.

\begin{ques}\label{rigidityq}
Let $(M,g)$ be a closed negatively curved 3-manifold containing a sequence of distinct asymptotically geodesic surfaces. Must $(M,g)$ have constant curvature?
\end{ques}

With a strong topological assumption on the surfaces, we are able to give a positive answer.  Recall that by geometrization every closed negatively curved 3-manifold carries a hyperbolic metric, which by Mostow rigidity is unique up to isometry.  

\begin{thm} \label{asymprigidity}
Let $(M,g)$ be a closed negatively curved 3-manifold and $(\Sigma_n)\se (M,g)$ be a sequence of asymptotically geodesic surfaces.

Suppose there is a diffeomorphism $\phi:M\to M$ so that $S_n:= \phi(\Sigma_n)$ form a sequence of asymptotically geodesic surfaces for the hyperbolic metric $g_{hyp}$ on $M$.  Then, $g$ is isometric to $g_{hyp}$.
\end{thm}

To prove this, we show that $\Gr M$ has a dense totally geodesic plane for the metric $g$, and by a classical result of Cartan we are able to conclude $g$ is hyperbolic. To do so, we use the fact that the surfaces $S_n$ are asymptotically dense in $\Gr M$, per Corollary \ref{dense}.  We are then able to show that the corresponding surfaces in the metric $g$ are asymptotically dense by following ideas from the article of Calegari-Marques-Neves \cite{calegari2022counting}, which proved a rigidity result for negatively curved closed 3-manifolds in terms of the growth rate of essential minimal surfaces, counted by area.

Question \ref{rigidityq} is not fully settled even when we assume $(M,g)$ contains a sequence of distinct \emph{totally geodesic} surfaces. In that setting, however, there has been much recent progress: Filip-Fisher-Lowe \cite{filip2024finiteness} answer it positively when the metric $g$ is analytic, and by the work of Bader-Fisher-Miller-Stover \cite{BFMS1} and Margulis-Mohammadi \cite{MohammadiMargulis2022_Arithmeticity-of-hyperbolic-3-manifolds-containing-infinitely-many-totally}, one concludes that $(M,g)$ is in fact an arithmetic hyperbolic manifold.

In all of these works, the fact that the distinct totally geodesic surfaces are asymptotically dense in $\Gr M$ (in fact, they equidistribute) plays a major role. A natural strengthening of the question is whether one can obtain rigidity theorems knowing only about finitely many or even a single surface. For example, one can ask whether $(M,g)$ containing a closed sufficiently large geodesic hyperbolic surface  is enough to force $g$ to be hyperbolic. Without any additional conditions, it is easy to show there can be no rigidity by perturbing the metric in a small ball. However, imposing an upper bound on the curvature we are able to prove the following theorem.  

\begin{thm}\label{bigsurface}
Let $(M,g_{hyp})$ be a closed hyperbolic 3-manifold. Then there is $\epsilon>0$ and a finite collection of closed totally geodesic surfaces so that if $S$ is $(1+\epsilon)$-quasi-Fuchsian surface that does not belong to the finite collection of totally geodesic surfaces, the following holds.  

Assume $g$ is a metric on $M$ with sectional curvature bounded above by -1 so that
\begin{itemize}
\item $\Sigma$ is a totally geodesic hyperbolic surface in $(M,g)$

\item There is a homotopy equivalence $\phi:M\to M$ so $\phi(\Sigma) = S$.
\end{itemize}
Then, $g$ is isometric to $g_{hyp}$ and $S$ is homotopic to a totally geodesic surface in $(M,g_{hyp})$.
\end{thm}

The second author proved this result for all but finitely many totally geodesic surfaces in $(M,g_{hyp})$, and building on work by V. Lima gave examples to show that it was necessary to exclude finitely many totally geodesic surfaces \cite{lowe2023rigidity}, \cite{lima2019area}. By the work described above we can only have infinitely many distinct closed totally geodesic $S$ when $M$ is arithmetic, which is a very strong condition on a hyperbolic 3-manifold. However, every closed hyperbolic 3-manifold contains many sequences of distinct asymptotically geodesic surfaces, from the Kahn-Markovi\'{c} surface subgroup theorem.  Therefore by the previous theorem every closed hyperbolic 3-manifold contains many essential surfaces that are not homotopic to totally geodesic hyperbolic surfaces in any metric with sectional curvature at most $-1$.  

The second author's original result was proved by arguing that the surface $\Sigma$ must be \emph{well-distributed}, meaning that any point in the universal cover $\tilde{M}$ is enclosed in a cube whose faces are lifts of $\Sigma$ to $\tilde{M}$ (compare with the notion of a filling surface recently introduced by X.H. Han \cite{hans}.) Then he shows that the existence of a well-distributed surface $\Sigma\se M$ implies the metric $g$ must be hyperbolic. To show $\Sigma$ is well-distributed, he uses the fact that the $S_k$ are asymptotically dense in $\Gr_2 M$, from the work of Mozes-Shah \cite{MS}. Corollary \ref{dense} implies the $S_k$ are asymptotically dense in $\Gr_2 M$ even if they are only asymptotically geodesic, so we are able to follow Lowe's original proof and obtain Theorem \ref{bigsurface}.

Another natural question in this direction is: 

\begin{ques}
Are the hypotheses of Theorem \ref{bigsurface} sufficient to guarantee that $(M,g_{hyp})$ is arithmetic, and so has infinitely many distinct finite area geodesic surfaces?  If a closed hyperbolic 3-manifold contains a closed well-distributed totally geodesic surface, then must it be arithmetic?  
\end{ques}

\subsection{Rigidity of totally umbilic surfaces in negative curvature}

We say a surface is \emph{totally umbilic} if all of its principal curvatures have the same positive constant value at every point. They are to constant mean curvature surfaces as minimal surfaces are to totally geodesic surfaces. A final variation on Question \ref{rigidityq} we consider is

\begin{ques}\label{intro:questiontotallyumbilic}
Let $(M,g)$ be a Riemannian 3-manifold with sectional curvature bounded above by $-c<0$ containing infinitely many distinct totally umbilic surfaces with mean curvature smaller than $c$. Must $(M,g)$ be hyperbolic?
\end{ques}

If the mean curvature of the totally umbilic surfaces is allowed to be large, then one does not expect a rigidity statement due to the fact that metric spheres in hyperbolic space are totally umbilic.  It suffices to take a hyperbolic metric and perturb it in a small neighborhood to obtain a variably negatively curved manifold with infinitely many totally umbilic hypersurfaces.  The condition that the mean curvature be smaller than $c$ is a natural choice because for every $\epsilon>0$ there are closed negatively curved 3-manifolds with sectional curvature non-constant and at most $-c$, and that contain infinitely many totally umbilic surfaces with mean curvature at most  $c + \eps$. These can be constructed from closed hyperbolic 3-manifolds with arbitrarily large injectivity radius, by perturbing the metric slightly outside of a large injectivity ball and then rescaling the metric  depending on the value of $c$.

We are able to prove the following theorem, which is analogous to Theorem \ref{asymprigidity}, and which we view as evidence for a positive answer to Question \ref{intro:questiontotallyumbilic} in dimension three. 

\begin{thm} \label{intro:totallyumbilic}
Suppose that $(M,g)$ is a closed negatively curved 3-manifold, and that it contains infinitely many distinct totally umbilic surfaces $\Sigma_n$ as in Question \ref{intro:questiontotallyumbilic}.  Suppose there is a homotopy equivalence $h:M\to M$ so that the $S_n:= h(\Sigma_n)$ are homotopic to a sequence of asymptotically totally umbilic surfaces for a hyperbolic metric $g_{hyp}$ on $M$.  Then $(M,g)$ must have constant curvature.

\end{thm}

Totally umbilic surfaces lack the connection to the geodesic flow that totally geodesic surfaces have, and we thus require a different argument from Theorem \ref{asymprigidity}.  The proof of Theorem \ref{intro:totallyumbilic} instead uses the Labourie-Smith theory of \textit{k-surfaces}, or surfaces for which the product of the principal curvaures, or extrinsic curvature, is constant.  Labourie introduced k-surfaces as a higher dimensional analogue of the geodesic flow and proved a number of deep results in this direction \cite{labourie1999lemme}, \cite{labourie2005random} (see \cite{labourie2002phase} for a survey.)  Their strong rigidity and uniqueness properties allow us to move information from constant to variable negative curvature and thereby obtain the control on the $S_n$ we need to prove Theorem \ref{intro:totallyumbilic}.

\subsection{Beyond asymptotically Fuchsian}

Very little is understood of the behavior of sequences of quasifuchsian minimal surfaces that fail to be asymptotically  geodesic, although recent work by Rao constructed many such examples \cite{rao2023subgroups}.  As a first step in this direction, we prove the following theorem in Section \ref{section:notgeodesic}.  

\begin{thm} \label{lastthmintro}
There are asymptotically dense sequences of closed quasifuchsian minimal or pleated surfaces that are not asymptotically Fuchsian. 
\end{thm}

\subsection{Related work and context}

The question of whether what is true for totally geodesic surfaces in hyperbolic 3-manifolds remains true for sequences of asymptotically geodesic minimal surfaces has been a focus of recent activity \cite{a}, \cite{calegari2022counting}, \cite{kahn2023geometrically}, \cite{labourie2022asymptotic}.

More generally, Thurston noticed that closed surfaces with principal curvatures smaller than one are incompressible (see \cite{leininger2006small} for a proof), and since then surfaces with small principal curvatures have been important in hyperbolic 3-manifold theory. Recent papers in which they play an important role include \cite{brody2024approximating},\cite{bronstein2023almost}, \cite{huang2021beyond},  \cite{marques2024conformal}, \cite{nguyen2025weakly}.

Work by Solan studied subgroups of hyperbolic 3-manifold groups that are asymptotic to Fuchsian subgroups in a different sense-- that their critical exponents tended to one from below-- and proved rigidity and gap theorems in that context \cite{solan2024critical}.  

The second author proved density theorems for minimal surfaces in certain negatively curved 3-manifolds \cite{lowe2021deformations}. The techniques of that paper could be combined with Corollary \ref{dense} of this paper to prove asymptotic density statements for sequences of minimal surfaces in certain negatively curved 3-manifolds, provided that the minimal surfaces in the sequence were homotopic to asymptotically Fuchsian surfaces in the hyperbolic metric.

Arguments similar to Section \ref{section:umbilic} have prior to this been used to prove various asymptotic rigidity statements for k-surfaces \cite{alvarez2022foliated}, \cite{alvarez2024rigidity} (see \cite{alvarez2025foliated} for a survey.)

Finally we mention again the work by X.H. Han that obtains results similar to Theorem \ref{contra} and its corollaries, and gives applications inspired by analogies to filling curves on surfaces \cite{hans}.

\subsection{Outline}
In Section \ref{section:contra} we give the proof of Theorem \ref{contra}. In Section \ref{section:gap} we prove Theorem \ref{intro:acylindricalgap}. In Section \ref{section:possacc} we establish some structural facts regarding accumulation sets of asymptotically geodesic surfaces in hyperbolic manifolds.  In Section \ref{section:higherintro} we prove Theorem \ref{higherintro}.  In Section \ref{section:totallygeodesicvariablecurvature} we prove Theorems \ref{asymprigidity} and \ref{bigsurface}.  In Section \ref{section:umbilic} we prove Theorem \ref{intro:totallyumbilic}.  In Section \ref{section:notgeodesic} we prove Theorem \ref{lastthmintro}.

\subsection{Acknowledgements}

The second author was supported by NSF grant DMS-2202830.  The second author thanks Graham Smith for answering some questions about k-surfaces.  

\section{Proof of Theorem \ref{contra}} \label{section:contra}

Let $M$ be a finite volume hyperbolic 3-manifold. Denote by $\Gr M$ the Grassmann bundle of tangent 2-planes to $M$.  We start by recording a useful lemma.

\begin{lem}\label{isolate}
Let $E\se \Gr M$ be a finite collection of finite area totally geodesic surfaces. Then there is $\delta = \delta(E)>0$ so that if $S\se M$ is an essential closed immersed surface that is not homotopic to any totally geodesic surface and no essential loop of which is homotopic into a cusp, then $d(\Pi,E)>\delta$ for some tangent plane $\Pi$ to $S$ in the $\delta$-thick part of $M$.  
\end{lem}

We first prove the following.  

\begin{claim}\label{homotope}
Let $G\se \Gr M$ be a finite collection of closed (compact) totally geodesic surfaces. There is $\delta = \delta(G)>0$ so that if $S\se M$ is a closed essential immersed surface whose tangent planes are contained in the $\delta$-neighborhood of $G$, then $S$ is homotopic to a cover of a surface in $G$.
\end{claim}

\begin{proof}
Note that as subsets of $\Gr M$, closed totally geodesic surfaces of $M$ are always embedded in $\Gr M$. In particular, distinct surfaces in $E$ are disjoint as subsets of $\Gr M$.

We may first take $\delta>0$ small enough so the $\delta$-neighborhoods of the components of $G$ are all disjoint from each other in $\Gr M$. Since $S$ is connected, we may assume it lies in the $\delta$-neighborhood of a component $T$ of $G$. 

Thus, for each point $p\in S$, there is a unique nearest point $\pi(p)\in T$. We can construct a homotopy having each point $p\in S$ traverse the geodesic segment in $M$ from $p$ to $\pi(p)$. Provided $\delta$ was chosen small enough, the map $\pi$ is a local diffeomorphism, and so a covering map. We conclude $S$ is homotopic to a cover of $T$.
\end{proof}

We now give the proof of Lemma \ref{isolate}. 
\begin{proof}
We choose $\delta$ smaller than the Margulis constant and the $\delta$ given by previous claim applied to the closed surfaces in $E$. We also require that the closed surfaces in $E$ are contained in the $\delta$-thick part of $M$.  By the previous claim and our choice of $\delta$, $S$ must have a tangent plane at a distance of at least $\delta$ from every tangent plane to every closed totally geodesic surface in $E$. It is therefore enough to show that $S$ has some tangent plane at a distance of at least $\delta$ from every tangent plane to a noncompact surface in $E$.  Assume for contradiction that this is not the case.

After perturbing $S$ slightly we can ensure that it intersects the boundary of the $\delta$-thin part $M^{\leq \delta}$ transversely.  Then every connected component of the intersection of $S$ with $M^{\leq \delta}$ must be homeomorphic to a disk or an annulus, since $S$ is essential.  Choosing $\delta$ smaller if necessary, reasoning as in Claim \ref{homotope} we can show that every connected component of the intersection $S \cap M^{\geq \delta}$ must be contained in a small neighborhood of at most one finite area totally geodesic surface in $E$.  The only way that $S \cap M^{\geq \delta}$ could be contained in a union of at least two distinct finite area totally geodesic surfaces (while not being contained in just one) would be if one of the annuli components of the intersection of $S$ with the thin part included to a parabolic element of $\pi_1(M)$-- such an annulus would have one boundary circle on a connected component of a component of $S \cap M^{\geq \delta}$ near one finite area totally geodesic surface and one boundary circle on a component of $S \cap M^{\geq \delta}$ near another finite area totally geodesic surface--  which is contrary to our assumption that $S$ contained no essential loops that could be homotoped into a cusp.   

Therefore every connected component of the intersection of $S$ with the thin part is contained in a small neighborhood of the same finite area surface $\Sigma_0$, and the core curves of all annuli in the intersection of $S$ with the thin part are contractible in $S$. In the same way as in the proof of Claim \ref{homotope} we can then homotope $S$ to a new surface $S'$ so that $S' \cap M^{\geq \delta}$ is contained in $\Sigma_0$.  Therefore for some choice of basepoint on $\Sigma_0$ the fundamental group of $S'$ injectively includes to a subgroup of the fundamental group of $\Sigma_0$, which is impossible because $\Sigma_0$ is cusped.

Therefore $S$ has tangent planes at a distance of at least $\delta$ to all tangent planes to non-compact surfaces in $E$, which completes the proof.   

\end{proof}

Below, for $p\in \Gr M$, denote by $\Delta_R(p)$ denote the geodesic disc of radius $R$ tangent to $p$. For a metric space $X$, we let $N_{\delta} A$ denote the $\delta$-neighborhood of subset $A\se X$. We say $A\se X$ is $\eps$-\emph{dense} if $N_{\eps} A = X$ -- in other words, $A$ intersects every ball of radius $\eps$ in $X$.
\iffalse
We say that a subset of $\Gr M$ is $\eps$\textit{-dense} if it intersects every ball of radius $\eps$ in $\Gr M$.
\fi

\subsection{When $M$ has finitely many finite area totally geodesic surfaces}

We first handle the case that $M$ has only finitely many finite-area totally geodesic surfaces. Throughout this bit, let $E$ denote the union of all the finite-area totally geodesic surfaces of $M$.

\begin{lem}\label{epsdense}
Assume that $M$ is compact. Then for every $\eps>0$ and every $\delta>0$, there exists $R_0(\eps,\delta)$ so if $R>R_0$ and $d(p,E)\geq \delta$, then $\Delta_R(p)$ is $\eps$-dense in $\Gr M$.
\end{lem}

\begin{proof}
If the statement is false, then there is $\eps,\delta>0$ and a sequence of $R_j$ tending to infinity so that for every positive integer $j$, there is $p_j\notin N_{\delta}(E)$ so that $\Delta_{R_j} (p_j)$ is not $\eps$-dense in $\Gr M$. In particular, $\Delta_{R_j} (p_j)$ misses an $\eps$-ball $B_{\eps} (q_j)$ around some $q_j \in \Gr M$.

From the compactness of $\Gr M - N_{\delta}(E)$ and $\Gr M$, we can extract common convergent subsequences of $(p_j)$ and $(q_j)$ converging to some $p\notin N_{\delta} (E)$ and $q\in \Gr M$, respectively. For simplicity, we also denote these subsequences by $(p_j)$ and $(q_j)$.

By construction, $\Delta_{R_j}(p_j)$ does not intersect $B_{\eps}(q_j)$. For sufficiently large $j$, we can also ensure that $\Delta_{R_j}(p_j)$ does not intersect $B_{\eps/2}(q)$.

In particular, if we fix a (sufficiently large) $k$, then for $j\geq k$, $\Delta_{R_k}(p_j)$ does not intersect $B_{\eps/2}(q)$. But as $k$ is fixed, the Hausdorff distance between $\Delta_{R_k}(p_j)$ and $\Delta_{R_k}(p)$ goes to zero as $j\to \infty$. Thus for $j$ large enough, we have that $\Delta_{R_k}(p) \se N_{\eps/4} \Delta_{R_k} (p_j)$. We conclude $\Delta_{R_k}(p)$ does not intersect $B_{\eps/4}(q)$ and so the geodesic plane $\Delta(p) := \bigcup_{k\geq 1} \Delta_{R_k}(p)$ through $p$ is not dense in $\Gr M$.

This contradicts the theorem of Ratner and Shah (\cite{R},\cite{S}) which tells us that the geodesic plane through any point $p\in \Gr M - E$ is dense in $\Gr M$.
\end{proof}

If $M$ has cusps, essentially the same argument shows the following.  

\begin{lem}\label{epsdensecusps}
For every $\eps>0$, $\delta>0$ and $\eta>0$, there exists $R_0(\eps,\delta,\eta)$ so that if $R>R_0$, $p \in M^{\geq \eta}$ and $d(p,E)\geq \delta$, then $\Delta_R(p)$ is $\eps$-dense in $\Gr M^{\geq \eta}$, where $M^{\geq \eta}$ is the $\eta$-thick part of $M$.
\end{lem}

We can now prove Theorem 1.1 when $M$ has only finitely many finite area totally geodesic surfaces. 

\begin{proof}[Proof of Theorem 1.1, special case]
First assume that $M$ is compact. Let $\eps>0$ be the radius of $B$ and let $\delta>0$. By Lemma \ref{epsdense}, there is $R_0>0$ so whenever $d(p,E)>\delta$, then $\Delta_R(p)$ is $\eps$-dense in $\Gr M$.

Let $S$ be a $(1+\xi)$-quasifuchsian minimal surface distinct from the surfaces in $E$. A result of Seppi tells us the absolute values of the principal curvatures of $S$ are uniformly bounded by $C\log(1+\xi)$, where $C$ is a universal constant \cite{Se}. 
In particular, if we denote by $D^S_R(p)$ the intrinsic disc in $S$ of radius $R$ passing through $p$, for $\xi$ sufficiently small depending on $\eps$, we have that
\[
\Delta_R(p) \se N_{\eps} (D^S_R (p)).
\]

From Lemma \ref{isolate}, we know that every quasifuchsian surface $S\se \Gr M$ has at least one point $p(S) \in \Gr M$ outside of $N_{\delta} E$. In particular, $D^S_R (p(S))$ is $2\eps$-dense in $\Gr M$.

If $M$ has cusps then one can apply Lemma \ref{isolate} and Lemma \ref{epsdensecusps} to argue in the same way, replacing $M$ with the $\eta$-thick part of $M$ in the relevant places.

\end{proof}

\subsection{When $M$ has infinitely many finite area totally geodesic surfaces}

\begin{lem}\label{epsdensearith}
Assume that $M$ is compact. Then for every $\eps>0$, there exists a collection $E\se \Gr M$ of finitely many closed totally geodesic surfaces such that for every $\delta>0$, there is $R_0>0$ so if $R>R_0$ and $d(p,E)\geq \delta$ then $\Delta_R(p)$ is $\eps$-dense in $\Gr M$.
\end{lem}

\begin{proof}
Let $\eps>0$. We choose $E$ to be the union of all the totally geodesic surfaces of $M$ which are not $\eps/4$-dense in $\Gr M$. From the equidistribution theorem of Mozes and Shah (\cite{MS}), we know $E$ is the union of finitely many surfaces. 

Now fix $\delta>0$. We now wish to show there is $R_0 = R_0(\eps,\delta,E)$ so whenever $p\notin N_{\delta}(E)$, then $\Delta_R (p)$ is $\eps$-dense in $\Gr M$. If we suppose that is not true, then we may find, as in the proof of Lemma \ref{epsdense}, a point $p\in \Gr M - N_{\delta}(E)$ so the geodesic plane $\Delta(p)$ through $p$ misses a ball of radius $\eps/4$ around some point $q$. If $p$ does not lie in a totally geodesic surface, then this contradicts the Ratner-Shah theorem on geodesic planes, as explained before. If $p$ does lie in a totally geodesic surface, then it has to be a surface not in $E$, which by construction is $\eps/4$-dense in $\Gr M$, again a contradiction.
\end{proof}

\begin{proof}[Proof of Theorem \ref{contra} in general]
Let $\eps>0$, and assume that $M$ is compact. By Lemma \ref{epsdensearith}, there is a finite collection $E\se \Gr M$ of geodesic surfaces, $\delta > 0$ small enough so $\Gr M - N_{\delta}(E)$ is nonempty, and $R_0>0$ so that for every $R>R_0$, if $p\in \Gr M - N_{\delta}(E)$ then $\Delta_R(p)$ is $\eps$-dense in $\Gr M$.

Now, by the same arguments (the last three paragraphs) from the proof of Theorem \ref{contra} in the previous case, we can show that for $\xi$ sufficiently small, $(1+\xi)$-quasifuchsian surfaces distinct from those contained in $E$ are $2\eps$-dense in $\Gr M$.

The arguments in the case that $M$ has cusps are again similar.  
\end{proof}

\section{Nonexistence of asymptotically geodesic surfaces in infinite volume} \label{section:gap}

In this section we show that, unlike in the finite-volume case, there is a gap between totally geodesic surfaces and non-totally geodesic essential surfaces for geometrically finite hyperbolic 3-manifolds of infinite volume.

\begin{thm}
Let $M=\Gamma\backslash \bH^3$ be an infinite volume  geometrically finite hyperbolic 3-manifold. 
Then there is $\eps = \eps(\Gamma)> 0$ so all $(1+\eps)$-quasifuchsian closed surface subgroups of $\Gamma$ are Fuchsian. 
\end{thm}

The same proof also gives the following theorem.  
\begin{thm}
Let $M=\Gamma\backslash \bH^3$ be an infinite volume geometrically finite hyperbolic 3-manifold.  Then there is $\eps = \eps(M)>0$ so that if $S\se M$ is a closed  surface with principal curvatures uniformly bounded by $\eps$, then $S$ is totally geodesic.
\end{thm}

\begin{rem}
We expect that similar statements could be obtained in higher dimensions for maximal surfaces (i.e., surfaces that cannot be homotoped into a proper totally geodesic submanifold) modulo a version of the Benoist-Oh theorem in higher dimensions (see \cite{lee2024orbit} for results in this direction in the higher dimensional case, assuming Fuchsian ends) but we focus here on the three-dimensional case. The question of whether there are convex cocompact hyperbolic $n$-manifolds without Fuchsian ends for $n>3$ is still unresolved, although Kerckhoff-Storm gave a negative answer locally \cite{kerckhoff2012local}. 
\end{rem}

\begin{proof}
{\bf I. A proof when $\core M$ is compact and has totally geodesic boundary.}

We first give a proof in the special case that $\core M$ is compact and has totally geodesic boundary. Let $\hat{M}$ be the closed hyperbolic 3-manifold obtained by doubling $\core M $

Suppose $(Q_i) \leq \Gamma$ is a sequence of asymptotically Fuchsian surface subgroups. They give rise to a sequence of immersed minimal surfaces $f_i: S_i\to M$ with $(f_i)_*(\pi_1 S_i) = Q_i$ \cite{U}, and by the maximum principle $f_i(S_i)$ always lies in the convex core of $M$. (The $S_i$ are obtained by minimizing area in the covering spaces corresponding to the injective inclusions of their fundamental groups.)  In addition, $f_i: S_i\to \hat{M}$ is a sequence of asymptotically geodesic minimal surfaces by \cite{Se}.  These surfaces are not dense in $\hat{M}$, as they never touch $\hat{M} - M$. This contradicts Theorem \ref{dense}. 

\medskip

{\bf II. The proof in general. }

The crucial ingredient in the proof is Theorem 1.5 of \cite{BO} that states that for a geometrically finite infinite volume hyperbolic 3-manifold $M$ there are only finitely many finite area totally geodesic surfaces contained in $\core M$, and these are the only closed immersed (not necessarily compact) totally geodesic planes in $\core M$.  We denote their union, as a subset of $\Gr_2 M$, by $E$.

Suppose $(Q_i)\leq \Gamma$ is a sequence of asymptotically Fuchsian surface subgroups. As before they give rise to a sequence of minimal surfaces $S_i \se \core M$ with $\pi_1 S_i = Q_i$ such that the $L^{\infty}$ norms of the principal curvatures of $S_i$ tend to zero as $i\to\infty$.    

If the $S_i$ were not  totally geodesic, then the same arguments as the proof of Claim \ref{homotope} and Lemma \ref{isolate} show that there is $\delta>0$ and $p_i \in S_i$ so that the tangent plane $\Delta(p_i)$ to $S_i$ at $p_i$ is at a distance of at least $\delta$ from $E$ in $\Gr_2 M$ and contained in the $\delta$-thick part of $\core M$.

The proof of the following statement can be proved by an argument by contradiction similar to the proof of Lemmas \ref{epsdense} and \ref{epsdensecusps}, using the fact from Benoist-Oh (\cite{BO}[Theorem 1.5]) that every immersed totally geodesic plane not contained in $E$ leaves the convex core of $M$.  

\begin{claim}
There is $R= R(\delta)>0$ so that for every $p$ in the intersection of the $\delta$-thick part of $M$ with $\Gr_2 (\core M) - N_{\delta} E$, the totally geodesic disc $\Delta_R(p)$ of radius $R$ based at $p$ leaves $\core M$.
\end{claim}

The fact that the principal curvatures of $S_i$ are going to zero in $L^{\infty}$ norm tells us that the Hausdorff distance between the intrinsic discs $D_R(p_i)\se S_i$ of radius $R$ centered at $p_i$ and the geodesic discs $\Delta_R(p_i)$ goes to zero as $i\to\infty$. In particular, $\Delta_R(p_i)$ leaves $\core M$ for large enough $i$, contradicting that the surfaces $S_i$ all lie in $\core M$.

\end{proof}
\section{Possible accumulation sets in higher dimensions} \label{section:possacc}

Let $M$ denote a closed $d$-dimensional hyperbolic manifold.

Let $(X_n)$ be a sequence of sets in the Grassman bundle $\Gr M= \Gr_2 M$ of tangent 2-planes to $M$. The \emph{accumulation set} of $(X_n)$ consists of the points $x\in \Gr M$ arising as limits of sequences $x_n \in S_n$. Note that this is a strictly stronger condition than if we required only that the points in the set arose as subsequential limits of sequences $x_n$. Note also that $(X_n)$ is asymptotically dense exactly when its accumulation set is all of $\Gr M$ and $(X_n)$ is always asymptotically dense in its accumulation set.  Finally, we point out that the accumulation set of a sequence that converges in the Hausdorff topology is the Hausdorff limit of that sequence.

The goal of this section is to prove the following theorem.  

\iffalse
\begin{thm}
Let $(S_n)$ be a sequence of asymptotically geodesic closed connected surfaces in $M$. Then the Hausdorff limit of the $(S_n)$ in $M$ is a finite union of totally geodesic submanifolds of $M$ of dimension $3\leq k \leq d$.
\end{thm}
\fi

\begin{thm}
Let $S_n$ be a sequence of asymptotically Fuchsian minimal surfaces in $M$. Then, the accumulation set of $S_n$ is a finite union $E$ of totally geodesic submanifolds of $M$ of dimension $3\leq k \leq d$. 
\end{thm}

\begin{rem}
The same proof shows that the statement of the theorem holds for a sequence $S_n$ of asymptotically geodesic surfaces, that are not assumed to be minimal.   
\end{rem}

\begin{lem}\label{higherd}
Suppose $S_n$ is a sequence of asymptotically Fuchsian minimal surfaces that is not contained in a finite union of closed totally geodesic submanifolds of $M$. Then $S_n$ is asymptotically dense. 
\end{lem}

\begin{proof}
The proof of the Lemma makes use of the Ratner-Shah theorem and closely follows the proof of Theorem \ref{dense} in three dimensions.

Following similar arguments to the proofs of Lemmas \ref{epsdense}, \ref{epsdensearith}, we can show: 

\begin{lem}\label{epsdensehigher}
For every $\eps>0$ and $\delta>0$, there is a finite union of totally geodesic submanifolds $E\se M$ with dimensions in $\{2,\ldots,d-1\}$ and $R_0 (\eps,\delta,E)>0$ so that if $R>R_0$ and $p$ is a tangent 2-plane to $M$ at a distance of at least $\delta$ from every tangent 2-plane to $E$, then $\Delta_R(p)$ is $\eps$-dense in $\Gr_2 M$.
\end{lem}

We can also prove the following lemma analogous to Lemma \ref{isolate}. 

\begin{lem}\label{higherisolate} 
Suppose $E\se  M$ is a finite union of proper totally geodesic submanifolds of $M$ and let $S$ be an essential surface of $M$ that is not homotopic to a surface contained in $E$. Then, there is $\eps=\eps(E)>0$ so that  there is a tangent 2-plane $p$ to $S$ at a distance of at least $\eps$ from every tangent 2-plane to $E$.
\end{lem}

\begin{proof}
It suffices to show that for $\eps>0$ small enough, if $d_{\Gr_2(M)}(S,E) \leq \eps$, then $S$ can be homotoped to a surface contained in $E$. To do so, as in the proof of Claim \ref{homotope}, we let $\pi:S\to E$ be the nearest point projection, which is well-defined if  $\eps>0$ was chosen small enough. We homotope $S$ into $E$ by having each point $x\in S$ traverse the geodesic joining $x$ to $\pi(x)$.
\end{proof}

We may now finish the proof of Lemma \ref{higherd}. Let $\eps>0$ and $\delta>0$. By Lemma \ref{epsdensehigher}, we may choose a finite union of totally geodesic submanifolds $E\se  M$ and and $R>0$ so that whenever $d_{\Gr_2 M}(p,E) > \delta$, then the geodesic disc $\Delta_R(p)$ is $\eps$-dense in $\Gr M$.

From Lemma \ref{higherisolate}, we know that there is a sequence of tangent planes $p_n\se S_n$ with $d_{\Gr_2 M}(p_n,E)\geq \delta$. Here we are using the fact that an essential minimal surface can be homotoped into a totally geodesic submanifold if and only if it is already contained in it.  

From the work of Jiang (section 2 of \cite{j}) we know that for $n$ large enough the $S_n$ are smooth surfaces without branch points, and that their principal curvatures converge uniformly to zero as $n\to\infty$. In particular, if we denote by $D^n_R(q)$ the intrinsic disc in $S_n$ of radius $R$ passing through $q$, we have that for $n$ large enough, $d_{\Gr_2 M}(D^n_R(p_n), \Delta_R(p_n))\leq \eps$.

Thus, there is $N$ so that for $n\geq N$, the intrinsic discs $D^n_R(p_n)$ are $\eps$-dense in $\Gr M$. In particular, $S_n$ are $2\eps$-dense for $n\geq N$.
\end{proof}

Lemma \ref{higherd} above shows that if $(S_n)$ are asymptotically Fuchsian minimal surfaces that are \emph{not} asymptotically dense, then their accumulation set is contained in a finite union $E$ of totally geodesic submanifolds. 

Denote by $\mathcal{A}$ the accumulation set of the $S_n$.  Let $\mathcal{G}$ be the set of all finite unions of totally geodesic submanifolds that contain $\mathcal{A}$.  We claim that the intersection $F$ of the sets in $\mathcal{G}$ is also an element of $\mathcal{G}$.  To see this, enumerate the elements of $\mathcal{G}$ as $g_1,g_2,..$.  We can write the intersection $F=\cap g_i$ as the intersection of the descending chain of sets $F_i:=\cap_{k=1}^i g_i$: 
\[
F=\cap_{i=1} ^\infty F_i. 
\]

We know that $F_{i+1} \subset F_i$, and that either $F_i=F_{i+1}$, or else the sum of the dimensions of the connected totally geodesic submanifolds of $F_{i+1}$ is smaller than those of $F_i$.  Therefore the sequence of $F_i$ is eventually constant, and equal to $F$.  

\subsection{Claim: $F=\mathcal{A}$}  Assume for contradiction that $\mathcal{A}$ is a proper subset of $F$, and let $\Sigma$ be a closed totally geodesic submanifold of $F$ that contains a point that is not contained in $\mathcal{A}$.  We know that $\mathcal{A} \cap \Sigma$ is not contained in a union of closed totally geodesic submanifolds of $\Sigma$, otherwise we could take $F$ to be a smaller set.  

If there were some tangent plane $\Pi$ to $\Sigma$ in $\mathcal{A}$ that was not contained in any closed totally geodesic submanifold of $\Sigma$, then the totally geodesic plane tangent to $\Pi$ would be dense in $\Sigma$ by Ratner's theorems.  Since $\Pi$ is in the accumulation set of the $S_n$ and the $S_n$ are asymptotically geodesic, this implies that all of $\Sigma$ is in the accumulation set of the $S_n$, which finishes the argument in this case.

Assume on the other hand that each tangent plane $\Pi$ in $\mathcal{A}$ to $\Sigma$  is contained in a closed totally geodesic $\Sigma_{\Pi}$ that is a proper submanifold of $\Sigma$, so that no closed totally geodesic submanifold of $\Sigma_{\Pi}$ contained the totally geodesic plane tangent to $\Pi$. Then since the totally geodesic plane tangent to $\Pi$ is dense in $\Sigma_{\Pi}$ by Ratner's theorems, in the same way as in the previous paragraph all of $\Sigma_{\Pi}$ must be contained in $\mathcal{A}$.  This shows that $\mathcal{A} \cap \Sigma$ is a union of closed totally geodesic submanifolds. 

Recall that a closed totally geodesic submanifold in $\mathcal{A} \cap \Sigma$ is \textit{maximal} if it is not contained in a closed totally geodesic submanifold of $\mathcal{A} \cap \Sigma$ of higher dimension. Then every closed totally geodesic submanifold of $\mathcal{A} \cap \Sigma$ is contained in a maximal totally geodesic submanifold of $\mathcal{A} \cap \Sigma$.  We claim that there are only finitely many maximal totally geodesic submanifolds of $\mathcal{A} \cap \Sigma$.  If not, then we can choose a sequence $\Sigma_1'$, $\Sigma_2'$,.. of them.  By Ratner's theorems, the $\Sigma_i'$ must be dense in a closed totally geodesic submanifold $\Sigma'$ of $\Sigma$.  Reasoning similar to before shows that $\Sigma'$ is contained in $\mathcal{A} \cap \Sigma$, which contradicts maximality of the $\Sigma_i'$.  

Therefore there are only finitely many maximal totally geodesic submanifolds of $\mathcal{A} \cap \Sigma$. Since $\mathcal{A} \cap \Sigma$ is a union of closed totally geodesic submanifolds and each closed totally geodesic submanifold of $\mathcal{A} \cap \Sigma$ is contained in a maximal such, $\mathcal{A} \cap \Sigma$ is equal to the union of the maximal totally geodesic submanifolds it contains, of which there are only finitely many.  Since $\mathcal{A} \cap \Sigma$ is not contained in finite union of closed totally geodesic submanifolds of $\Sigma$, the only possibility for $\mathcal{A} \cap \Sigma$ is that it is equal to $\Sigma$, which finishes the proof.

\section{Realizing accumulation sets in higher dimensions} \label{section:higherintro}

In this section, we will prove Theorem \ref{higherintro}, restated below for convenience.  

\begin{thm}\label{higher}
Suppose $M$ is a closed hyperbolic $d$-manifold containing two closed totally geodesic 3-dimensional submanifolds $N_1$ and $N_2$ that intersect along a closed totally geodesic surface.

Then, there exists a sequence $(S_n)$ of closed connected asymptotically Fuchsian pleated surfaces  whose accumulation set is $N_1\cup N_2$. There also exists a sequence $(S_n)$ of closed connected asymptotically geodesic minimal surfaces whose limit in the Hausdorff topology of $M$ is $N_1\cup N_2$.
\end{thm}

This is in contrast with the case where all the $S_n$ are totally geodesic. In that setting, using the work by Mozes-Shah \cite{MS} we can prove the following proposition.  

\begin{prop} \label{prop:totallygeodesicaccumulation}
Let $M$ be a closed hyperbolic $d$-manifold. The accumulation set of a sequence $S_n\se M$ of closed totally geodesic surfaces consists of a closed connected totally geodesic submanifold $S$ (i.e., a totally geodesic submanifold $S$ whose lift to the Grassmann bundle of tangent $dim(S)$ planes to $M$ is connected.)  
\end{prop}

\begin{proof}[Proof of Proposition \ref{prop:totallygeodesicaccumulation}]
Let $G = \SO_0(d,1)$ and $H = \SO_0(2,1) \leq G$.

For each $n$, let $\hat{S}_n$ be a closed $H$-orbit in $\Fr M = \Gamma\backslash G$ whose projection to $\Gr_2 M$ is equal to $S_n$. Let $\mu_n$ be the homogeneous probability measure supported on $\hat{S}_n$.

\begin{claim*}
Let $\pi: \Fr M \to \Gr_2 M$ denote the projection sending a frame $(p,v_1,\ldots,v_d)$ to $(p, \Span\{v_1,v_2\})$. Then, 
\[
\acc(\{S_n\}) = \pi \lef( \bigcap_{\mu} \supp(\mu) \ri),
\]
where $\mu$ ranges over all weak-* limits of subsequences of $(\mu_n)$.
\end{claim*}

The theorem of Mozes-Shah implies that each $\mu$ is the homogeneous measure supported on a closed orbit of some closed Lie subgroup $L_{\mu} \leq G$ properly containing $H$. In particular, the projection of $\cap_{\mu}\supp(\mu)$ to $\Gr_2 M$ is connected as the 2-plane Grassmann bundle to a totally geodesic submanifold $N\se M$ that corresponds to the connected subgroup of $G$ obtained as the intersection of the $L_{\mu}$, which proves the proposition assuming the claim.

We now prove the claim. 

Let $\mu_{n_k}\wkstar \mu$ denote a convergent subsequence of $(\mu_n)$. As $\mu_{n_k}$ is the homogeneous measure supported on $\hat{S}_{n_k}$, if $x_{n_k}\in \hat{S}_{n_k}$ is a convergent sequence, then $\lim_{k\to\infty} x_{n_k} \in \supp \mu$.  We are using here the fact from \cite{MS} that the closed subgroups $L_{\mu_{n_k}}$ corresponding to the $\mu_{n_k}$ are contained in the closed subgroup $L_{\mu}$ corresponding to $\mu$ for large enough $k$. Conversely, any point in $\supp \mu$ can be realized as a limit of $x_{n_k}\in \hat{S}_{n_k}$, or else there would be a neighborhood of that point to which $\mu$ assigned zero measure which is impossible. Therefore, $\acc(\hat{S}_{n_k}) = \supp\mu$. 

We conclude by noting that $\acc(\hat{S}_n)$ is the intersection of $\acc(\hat{S}_{n_k})$ over all subsequences $(\hat{S}_{n_k})$ such that $\mu_{n_k}$ is convergent.  That the former is contained in the latter is clear.  To see the other containment, assume for contradiction that there is a point $p$ in the intersection of $\acc(\hat{S}_{n_k})$ over all subsequences $(\hat{S}_{n_k})$ such that $\mu_{n_k}$ is convergent, but that $p$ is not contained in $\acc(\hat{S}_n)$. If this is the case, then there is some $\epsilon>0$ so that for any $N$, there is some $k>N$ so that every point on $\hat{S}_k$ is at a distance of at least $\epsilon$ from $p$.  Passing to a convergent subsequence of the probability measures corresponding to the $\hat{S}_k$, we obtain a weak-* limit measure that does not contain $p$, which is a contradiction.   
 \end{proof}

\subsection{Manifolds satisfying the hypotheses of Theorem \ref{higher}}

The lattices in hyperbolic space we consider will be arithmetic of simplest type.  We give now a  description of their basic properties and facts about them that will be relevant to our construction.  We follow closely the reference \cite{millson}[Section 2], see that paper for further details and references.

We give the construction only in the four-dimensional case, as the higher dimensional case is very similar. but  Let $q = x_1^2 + \cdots + x_4^2 - \sqrt{p} x_{5}^2$, for $p$ prime. Then, $ \SO(q,\bZ[\sqrt{p}])$ is a lattice in $\SO(q,\bR)$. A quadratic form is \textit{anisotropic} if it does not represent zero rationally.  Since $q$ is anisotropic, the lattice  $ \SO(q,\bZ[\sqrt{p}])$ is cocompact.  There is a real invertible matrix $T$ so that $T^t q T=q_{0}$, for $q_0$ the standard signature $(n,1)$ quadratic form.  Therefore $T^{-1} \SO(q,\bR) T = \SO(d,1)$, 
\[
\Gamma := T^{-1} \SO(q,\bZ[\sqrt{p}]) T \cap \SO_0(d,1)
\]
is a lattice in $\SO_0(d,1) = \Isom^+ \bH^d$, and $M:= \Gamma\backslash \bH^d$ is a closed hyperbolic $d$-manifold.

Let $q_1 = x_1^2 + x_2^2 + x_3^2 - \sqrt{p}x_{5}^2$ and $q_2 = x_2^2 + x_3^2 + x_4^2 - \sqrt{p}x_{5}^2$. Define \[\Gamma_i :=T^{-1} \SO(q_i, \bZ[\sqrt{p}]) T \cap \SO_0(d,1). \]
The groups $\Gamma_i$ can be seen as subgroups of $\Isom^+ H_i$, where $H_i$ is the hyperbolic 3-plane $\bH^d \cap \{x_1=0\}$ for $i=1$ and $\bH^d\cap \{x_2=0\}$ for $i=2$. In particular, $N_i:=\Gamma_i\backslash H_i$ are totally geodesic 3-suborbifolds of $M$. They meet orthogonally along the totally geodesic 2-orbifold $T = \Gamma\backslash (\bH^d \cap \{ x_1,x_2 = 0\})$.

The lattice $\Gamma$ may have torsion (i.e., $M$ may be an orbifold), but it is possible to pass to finite index subgroups that do not.  For a prime ideal $q$ in the ring of integers of $\bQ[\sqrt{p}]$, the congruence subgroup $\Gamma(q)$ of $\Gamma$ is defined as follows.  Note that all entries of each element of $\Gamma$ are algebraic integers in $\bQ[\sqrt{p}]$. We say that $\gamma \in \Gamma$ satisfies $\gamma \equiv 1 \text{ mod } q$ if each of its diagonal entries is congruent to 1 mod $q$, and each other entry off the diagonal is congruent to 0 mod $q$.  The matrices in $\Gamma(q)$ form a normal subgroup of $\Gamma$.  For all but finitely many $q$, $\Gamma(q)$ will be torsion free.   The lifts of $N_1$ and $N_2$ to $\Gamma(q) \backslash \mathbf{H}^d$ intersect in a totally geodesic surface each connected component of which finitely covers $N_1 \cap N_2$. This gives examples of manifolds satisfying the hypotheses of Theorem \ref{higher}.

\subsection{Proof of Theorem \ref{higher}} \label{section:pants}

\begin{figure}
\includegraphics[scale=0.8]{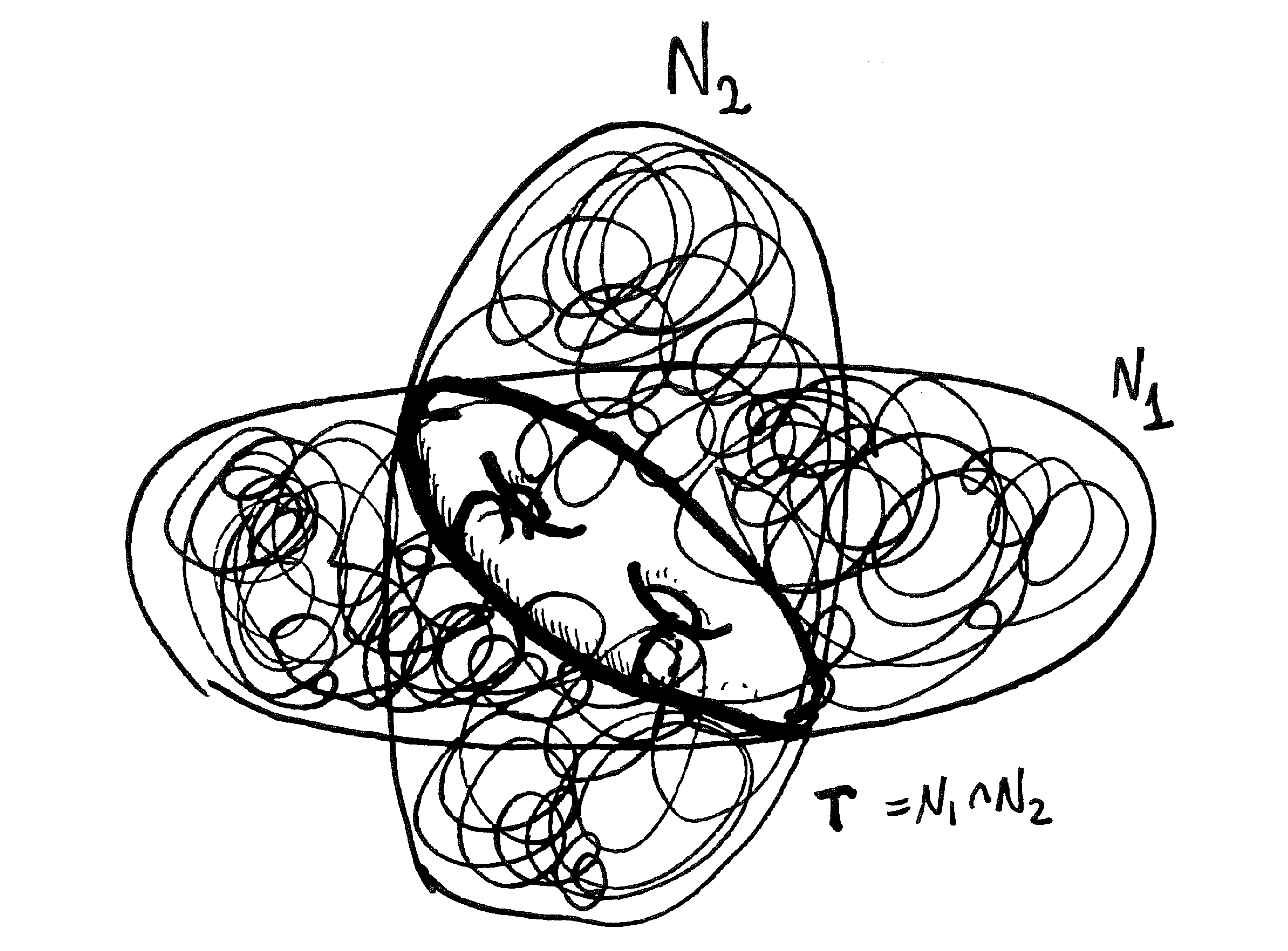}
\caption{The scrawly curve represents the surfaces $S(\eps_j,R_j)$, which become dense in $N_1\cup N_2$ as $j\to\infty$.}
\end{figure}

We follow the notation and terminology from \cite{a}. The $(\epsilon,R)$-good curves in M are the closed geodesics with complex translation length $2\epsilon$-close to
$2R$. The $(\epsilon,R)$-good
pants, are the maps $f : P \rightarrow M$ from a pair of pants $P$ taking the cuffs of $P$ to $(\epsilon,R)$-good curves.

From the work of Kahn-Wright \cite{KW}, given $\eps>0$ each submanifold $N_i$ contains a $(1+O(\eps))$-quasifuchsian pleated surface $S_i(\eps,R)$ that is made out of one copy of each $(\epsilon,R)$-good pants in the set of all  $(\epsilon,R)$-good pants in $N_i$, which we denote by $\pants(N_i)$.

Using the work of Liu-Marković \cite{LM}, one can show there are numbers $n_i(\eps,R)$ for $i=1,2$ so that there are \emph{connected} $(1+O(\eps))$-quasifuchsian surfaces, denoted also $S_i(\eps,R)$ made out of $n_i(\eps,R)$ copies of each good pants (see \cite{a}[Section 4].) 

Let $T$ be the totally geodesic surface along which $N_1$ and $N_2$ intersect. From the proof of the Ehrenpreis conjecture by Kahn-Marković \cite{KME}, we know that $T$ has a finite cover $\hat{T}$ admitting a pants decomposition $\Pi_T$ of $(\eps,R)$-good pants.

Choose a cuff $\gamma_1$ of $T$ that is part of $\Pi_T$ and let $\pi^-\in \pants^- (\gamma_1)$ and $\pi^+ \in \pants^+ (\gamma_1)$ be immersed pants in $T$ that are $(\eps,R)$-well glued along $\gamma_2$.

Using the fact that the pants in $\pants(N_1)$ are well-distributed along $\gamma_1$ (see Theorem 3.3 in \cite{KW}), we may find pants $p^-\in \pants^-(N_1)$ and $p^+\in \pants^+ (N_1)$ in $S_1(\eps,R)$ that are $(\eps,R)$-well glued along $\gamma_1$ and so that $\ft p^-$ is $\eps/R$-close to $\ft \pi^-$.

Thus we may cut $S_1$ and $T$ along $\gamma_1$ and glue $p^-$ to $\pi^+$ and $\pi^-$ to $p^+$ in a $(2\eps,R)$-good fashion. The reglued surface $S'$ will still be closed, connected and $1+O(\eps)$-quasifuchsian, by Theorem 2.2 from \cite{KW}.

We may choose another cuff $\gamma_2$ of $\Pi_T$ and similarly perform reglueings to join $S_2$ to $S'$ in a nearly Fuchsian way, giving us a new surface $S =S(\eps,R)$.

Taking a sequence $\eps_j \to 0$ and $R_j = R(\eps_j) \to \infty$, we obtain a sequence of asymptotically Fuchsian pleated surfaces $S(\eps_j,R_j)$ with accumulation set $N_1\cup N_2$.  It follows from \cite{j}[Section 2.3] that the $S(\eps_j,R_j)$ are homotopic to minimal surfaces $S_j$ without branch points.  Arguing as in \cite{j}  [Section 2.3, 3.2] and \cite{calegari2022counting}[Lemma 4.3], we can show that the $S_j$ have the same accumulation set as the $S(\eps_j,R_j)$.  To explain in more detail, one can argue by contradiction.  

Suppose that there is some $\epsilon>0$ and points $p_j$ on $S_j$ whose tangent planes are a distance of at least $\epsilon$ from any tangent plane to $S(\eps_j,R_j)$.  Choosing lifts $D_j$ of the $S_j$ to the universal cover that all intersect a fixed compact set containing lifts of $\tilde{p}$, we know that their limit sets are quasicircles that converge, up to passing to a subsequence, to a round circle that is the boundary at infinity of a totally geodesic disk $D$.  Jiang \cite{j} shows that the $D_j$ smoothly converge to $D$ on compact subsets. But it is also the case that the tangent planes to the corresponding lifts of the $S(\eps_j,R_j)$ Hausdorff converge to $D$ on compact subsets. Projecting $D_j$ and the lift of $S(\eps_j,R_j)$ back down to $M$ then gives a contradiction for large enough $j$.

\section{Rigidity of totally geodesic surfaces in negative curvature} \label{section:totallygeodesicvariablecurvature}
\subsection{} 
Given Corollary \ref{dense},  the proof of Theorem \ref{bigsurface} is essentially the same as the proof of \cite{lowe2023rigidity}[Theorem 1.5]. We restate Theorem \ref{bigsurface} for convenience.  

\begin{thm} 
Let $(M,g_{hyp})$ be a closed hyperbolic 3-manifold. Then there is $\epsilon>0$ and a finite collection of closed totally geodesic surfaces so that if $S$ is $1+\epsilon$-quasifuchsian surface that does not belong to the finite collection of totally geodesic surfaces, the following holds.  
Assume $g$ is a metric on $M$ with sectional curvature bounded above by -1 so that
\begin{itemize}
\item $\Sigma$ is a totally geodesic hyperbolic surface in $(M,g)$

\item There is a homotopy equivalence $\phi:M\to M$ so $\phi(\Sigma) = S$.
\end{itemize}
Then, $g$ is isometric to $g_{hyp}$ and $S$ is homotopic to a totally geodesic surface in $(M,g_{hyp})$.
\end{thm}

As before we can take $\Sigma$ to be the unique area-minimizing surface in its homotopy class and have principal curvatures as small as desired, by making $\epsilon$ small.  We say that a surface $\Sigma$ in $M$ satisfies the \textit{well-distribution property} if every point in $\tilde{M} \cong \mathbf{H}^3$ is contained in a solid cube whose faces are contained in lifts of $\Sigma$ to $\mathbf{H}^3$. It is a straightforward consequence of Theorem \ref{contra} that if $\epsilon$ is small enough and $\Sigma$ is $1+\epsilon$-quasifuchsian, then $\Sigma$ satisfies the well-distribution property unless it is totally geodesic.  We also know by the Mozes-Shah theorem that at most finitely many closed totally geodesic surfaces fail to satisfy the well-distribution property.

There is a slightly stronger version of the well-distribution property called the \textit{strong well-distribution property} \cite{lowe2023rigidity}[Definition 5.1]  Using the fact that the $\Sigma$ are asymptotically dense, it is not hard to check that $\Sigma$ satisfies the strong well-distribution property provided $\Sigma$ is $1+\epsilon$-quasifuchsian and not one of finitely many totally geodesic surfaces.

Assume that $h(\Sigma)$ is homotopic to a hyperbolic totally geodesic surface $\Sigma'$ in $N$.  The arguments in \cite{lowe2023rigidity}[Section 6.2] show that $\Sigma'$ satisfies the well-distribution property if $\Sigma$ satisfies the strong well-distribution property.  From there \cite{lowe2023rigidity}[Theorem 1.1] implies that $N$ has constant sectional curvature $-1$, and so by Mostow rigidity must be isometric to $M$.

\subsection{}

The proof of Theorem \ref{asymprigidity} requires a different argument. We restate it here for convenience. 

\begin{thm} 
Let $(M,g)$ be a negatively curved 3-manifold and $(\Sigma_n)\se (M,g)$ be a sequence of asymptotically geodesic surfaces.

Suppose there is a diffeomorphism $\phi:M\to M$ so that $S_n:= \phi(\Sigma_n)$ form a sequence of asymptotically geodesic surfaces for the hyperbolic metric $g_{hyp}$ of $M$.  Then, $g$ is isometric to $g_{hyp}$.
\end{thm}

Let $\Sigma_n$ be the sequence of asymptotically geodesic surfaces in a closed negatively curved manifold $M$, and let $h:M \rightarrow (M,g_{hyp})$ be the homotopy equivalence from the assumptions of the theorem.  

Let $\Sigma_n'$ be the sequence of asymptotically geodesic minimal surfaces in $(M,g_{hyp})$, obtained by taking the area-minimizing representatives of the homotopy classes $h(\Sigma_n')$, which are unique for large enough $n$.   In more detail, \cite{epstein1986hyperbolic}[Equation 5.5] implies that since the $h(\Sigma_n)$ are homotopic to a sequence of asymptotically geodesic surfaces, their quasiconformal constants tend to 1, and then work by Seppi \cite{Se} implies that area-minimizing representatives in their homotopy classes $\Sigma_n'$ are asymptotically geodesic and unique (see also \cite{huang2023uniqueness}.) 

Since the $\Sigma_n'$ are asymptotically Fuchsian and since it is not the case that all but finitely many are covering a fixed surface, we know that their tangent planes become dense in $Gr_2(M,g_{hyp})$ by Corollary \ref{dense}.  Choose tangent planes $\pi_n$ to $\Sigma_n'$ that converge, up to passing to a subsequence, to a tangent plane to $M$ that contains a tangent vector with dense orbit under the geodesic flow.

The homotopy equivalence between $(M,g)$ and $(M,g_{hyp})$ given by the assumptions of the theorem must be homotopic to a diffeomorphism by geometrization and the fact that every self homotopy equivalence of a closed hyperbolic manifold is homotopic to an isometry by Mostow rigidity:  fix such a diffeomorphism $F: (M,g_{hyp})\rightarrow (M,g)$ and choose a lift of $F$ to an equivariant diffeomorphism $\tilde{F}: \mathbf{H}^3 \rightarrow (\tilde{M},\tilde{g})$, where we have identified the universal cover of $(M,g_{hyp})$ with $\mathbf{H}^3$.  Choose lifts $\tilde{\Sigma}_n'$ of $\Sigma_n'$ to $\mathbf{H}^3$ whose intersection with a fixed fundamental domain for the action of $\pi_1(M)$ on $\mathbf{H}^3$ contain tangent planes that projects to $\pi_n$. After passing to a subsequence, we can assume that the $\tilde{\Sigma}_n'$  converge on compact subsets to a totally geodesic surface $\tilde{\Sigma}'$ that projects to a surface $\Sigma'$ in $(M,g_{hyp})$. Note that $\Sigma'$ contains the tangent plane $\pi$, and thus $\Sigma'$ has a tangent vector with dense orbit under the geodesic flow.

A choice of lift $\tilde{F}$ of the diffeomorphism $F$ to a map $(M,g_{hyp}) \rightarrow (M,g)$ between the universal covers and the choices of lifts $\tilde{\Sigma}_n'$ determine lifts $\tilde{\Sigma}_n$ of the $\Sigma_n$ to $(\tilde{M},\tilde{g})$, that are at finite Hausdorff distance from the $\tilde{F}(\tilde{\Sigma}_n')$.  Using the fact that the $\tilde{\Sigma}_n$ are asymptotically geodesic, and so in particular have uniformly bounded second fundamental forms, and arguing like in \cite{calegari2022counting}[Proposition 4.1], we can reason that, possibly after passing to a subsequence, the $\tilde{\Sigma}_n$ converge to a totally geodesic surface that we call $\tilde{\Sigma}$, and whose projection to $(M,g)$ we call $\Sigma$.  We claim that  there is a totally geodesic surface that is dense in the 2-plane Grassmann bundle of $(M,g)$.  It then follows that every tangent 2-plane to $M$ is tangent to a totally geodesic surface, which Cartan showed implies that $M$ must have constant sectional curvature \cite{cartan1928leccons}.  To finish the proof, it is therefore enough to show that there is a totally geodesic surface that is dense in the 2-plane Grassmann bundle of $(M,g)$.

There is a homeomorphism  $\Phi:\UT(M,g_{hyp}) \rightarrow \UT(M,g)$ that maps geodesics to geodesics (i.e., an orbit equivalence between the geodesic flow for $g$ and the geodesic flow for $g_{hyp}$) \cite{G}. By how $\Phi$ is constructed, it is homotopic to the lift $\UT(M,g_{hyp}) \rightarrow \UT(M,g)$ of a diffeomorphism $(M,g_{hyp}) \rightarrow (M,g)$.  By how we constructed $\Sigma'$ and $\Sigma$, we know that, possibly after composing $\Phi$ with an element of the finite group of isometries of $(M,g_{hyp})$, that $\Phi$ maps every unit tangent vector tangent to $\Sigma'$ to a unit tangent vector tangent to $\Sigma$.  Because $\Sigma'$ contains a unit tangent vector with dense orbit under the geodesic flow, the same must be true of $\Sigma$.  Choose a frame $\Pi$ in the frame bundle of $(M,g)$ that has dense orbit under the frame flow.  This is possible because the frame flow  of a closed negatively curved 3-manifold is ergodic \cite{brin1980ergodicity}. 

Because $\Sigma$ contains a vector with dense orbit under the geodesic flow, we can then choose a sequence of frames $\Pi_n$ on $\Sigma$ the initial vectors of which converge to the first vector of $\Pi$.  Passing to a convergent subsequence of the $\Pi_n$ so that they converge to some frame $\Pi'$, it will then be the case that the first two unit vectors in the frame $\Pi'$ span a tangent plane that is a limit of tangent planes to $\Sigma$, and thus is itself tangent to a totally geodesic surface.  Since $\Pi$ and $\Pi'$ share the same first vector and their second two vectors differ by a rotation in the plane orthogonal to the first vector, $\Pi'$ also has dense orbit under the frame flow. The tangent plane spanned by the first two vectors of $\Pi'$ is therefore tangent to a totally geodesic surface whose closure is all of $\Gr_2 M$, which finishes the proof.

\section{Rigidity of totally umbilic surfaces} \label{section:umbilic}

We give the proof of Theorem \ref{intro:totallyumbilic} from the introduction, restated here for convenience. 

\begin{thm} 
Suppose that $(M,g)$ is a closed negatively curved 3-manifold, and that it contains infinitely many distinct totally umbilic surfaces $\Sigma_n$ as in Question \ref{intro:questiontotallyumbilic}.  Suppose there is a homotopy equivalence $h:M\to M$ so that the $S_n:= h(\Sigma_n)$ are homotopic to a sequence of asymptotically totally umbilic surfaces for a hyperbolic metric $g_{hyp}$ on $M$.  Then $(M,g)$ must have constant curvature.  

\end{thm}

Let $M$ be a closed negatively curved 3-manifold with sectional curvature at most $-c$, and let $\Sigma_n$ be totally umbilic surfaces in $M$ each with mean curvature constant and equal to $H_n<c$. Let $h$ be a homotopy equivalence as in the statement of Theorem \ref{intro:totallyumbilic} between $M$ and a closed hyperbolic 3-manifold $(M,g_{hyp})$, so that the $h(\Sigma_n)$ are homotopic to a sequence of asymptotically totally umbilic surfaces $\hat{\Sigma}_n$.  Because all totally umbilic surfaces in $\mathbf{H}^3$ with principal curvatures less than 1 are constant distance surfaces to totally geodesic planes in $\mathbf{H}^3$, a simple contradiction compactness argument shows that the $\hat{\Sigma}_n$ are homotopic to a sequence $\Sigma_n'$ of asymptotically geodesic surfaces in $(M,g_{hyp})$.  The tangent planes to the $\Sigma_n'$ must be dense in $\Gr_2(M)$ by Corollary \ref{dense}, and consequently the same is true for the $\hat{\Sigma}_n$.  

\subsection{Background on k-surfaces}
A \textit{k-surface} is a surface for which the product of the principal curvatures, or extrinsic curvature, is constant and equal to $k$. They will be useful for us because totally umbilic surfaces are special cases of k-surfaces.  Labourie initiated the study of the dynamical properties of the space of k-surfaces, presenting them as a higher-dimensional analogue of the geodesic flow when the ambient manifold is negatively curved \cite{labourie1999lemme}, \cite{labourie2005random} (see \cite{labourie2002phase} for a survey.)  We describe in this section results that will play the role of the geodesic orbit equivalence result used in Section \ref{section:totallygeodesicvariablecurvature}.   

If $N$ is a negatively curved 3-manifold with sectional curvature at most $-k$, then the space of pointed k-surfaces in $N$ (the space of k-surfaces together with a choice of basepoint) suitably topologized has the structure of a \textit{lamination} $\mathcal{L}_{N}$, whose leaves are obtained by fixing a $k$-surface, and allowing the basepoint to vary \cite{labourie1999lemme}.  (While the k-surfaces corresponding to distinct leaves of this lamination may intersect inside $N$, they are disjoint in the phase space of pointed k-surfaces in $N$.)  Labourie in the local case and Smith in general proved that this lamination is independent of the metric on $N$ in the following sense: For any two $N_1$ and $N_2$ as above there is a leaf-preserving homeomorphism $\Phi_{\mathcal{L}}$ between $\mathcal{L}_{N_1}$ and $\mathcal{L}_{N_2}$ (see \cite{smith2021asymptotic} for precise statements.)  We will need the following two facts about the leaf-preserving homeomorphism $\Phi_{\mathcal{L}}$, which readily follow from the construction and basic properties of the $\Phi_{\mathcal{L}}$, and which we write as propositions.  
 
\begin{prop} \label{prop-phi1}
 Suppose that $f$ is a homotopy equivalence between $N_1$ and $N_2$.  Then we can choose $\Phi_{\mathcal{L}}$ so that for every $\pi_1$-injective $k$-surface $\Sigma$ in $N_1$, the k-surface corresponding to $\Sigma$ under $\Phi_{\mathcal{L}}$ is homotopic to $f(\Sigma).$
 \end{prop}

\begin{prop} \label{prop-phi2}
Suppose that $\Sigma_n$ is a sequence of k-surfaces in $N_1$, and that there is a k-surface $\Sigma$ in $N_1$ so that the following holds: there are lifts $\tilde{\Sigma}_n$ and $\tilde{\Sigma}$ to the universal cover of $N_1$ so that $\tilde{\Sigma}_n$ converges smoothly to $\tilde{\Sigma}$ on compact sets.  Then the same statement holds for the k-surfaces $\Sigma_n'$ and $\Sigma'$ in $N_2$ that correspond respectively to the $\Sigma_n$ and $\Sigma$ under $\Phi_{\mathcal{L}}$.  
\end{prop}

We will also use recent work by Alvarez-Lowe-Smith \cite{alvarez2022foliated}, that refined the existence of the leaf-preserving homeomorphisms $\Phi_{\mathcal{L}}$ as follows.  Let $(M,g_{hyp})$ be a closed hyperbolic 3-manifold.  Then for each $0<k<1$, the unit tangent bundle to $M$ has a natural foliation $\mathcal{F}_M(k)$ by totally umbilic surfaces with constant extrinsic curvature $k$, lifted by their unit normal vector fields (note that any totally umbilic surface has a canonical unit normal vector field, just by taking the unit vector field in the direction of the mean curvature vector.) This is because $\mathbf{H}^3$ has such a foliation that is invariant by isometries.  The following theorem is a consequence of \cite{alvarez2022foliated}[Theorem 2.1.3]. 

\begin{thm} \label{thm:ksurfaceconj}
For every negatively curved Riemannian 3-manifold $N$ with sectional curvature at most $-k$ the following holds. The unit tangent bundle $\UT(N)$ admits a foliation $\mathcal{F}_N(k)$ by k-surfaces lifted by their unit normal vectors.  The foliation  $\mathcal{F}_N(k)$ has the following property. Let $M$ be a hyperbolic manifold diffeomorphic to $N$. Then there is a homeomorphism $\Phi:\UT(M) \rightarrow \UT(N)$ sending leaves of $\mathcal{F}_M(k)$ to leaves of $\mathcal{F}_N(k)$.  Moreover, $\Phi$ can be chosen so that, for each leaf $F$ of $\mathcal{F}_N(k)$, $\Phi(F)$ and $\Phi_{\mathcal{L}}(F)$ are the same k-surface.    
\end{thm}

\subsection{}

As in the Section \ref{section:totallygeodesicvariablecurvature}, after passing to a subsequence we can find a dense immersed totally geodesic plane in $(M,g_{hyp})$ that lifts to a totally geodesic plane in $\mathbf{H}^3$ to which lifts of the $\Sigma_n'$ converge.  We can thus also find a totally umbilic plane $\hat{\Sigma}$ in $(M,g_{hyp})$ playing the same role for the $\hat{\Sigma}_n$: i.e., which has a lift to $\mathbf{H}^3$ to which lifts of the  $\hat{\Sigma}_n$ converge (this totally umbilic plane is at constant distance to the totally geodesic plane of the previous sentence.)   

After passing to a subsequence, we can assume that the $H_n$ converge to some real number $H$. The extrinsic curvatures of the totally umbilic surfaces $\Sigma_n$ in $N$ consequently converge to $H^2$.   Let $\Phi:\UT(M) \rightarrow \UT(N)$ be the homeomorphism from Theorem \ref{thm:ksurfaceconj} sending  tangential lifts of $k$-surfaces in $\UT(M)$ to tangential lifts of $k$-surfaces in $\UT(N)$ for $k=H^2$. Then $\Phi(\hat{\Sigma})$ is dense in $\UT(N)$, and projects to a k-surface $\Sigma$ in $N$ for $k=H^2$.  If $\Sigma$ is totally umbilic, then every tangent plane to $N$ is tangent to a totally umbilic surface.  That $M$ has constant curvature then follows from the main theorem of \cite{leung1971axiom}, which is the analogue for totally umbilic surfaces of the theorem by Cartan used in the previous section.  It is therefore enough to prove that $\Sigma$ is totally umbilic, which we will accomplish by showing every tangent plane to $\Sigma$ is a limit of tangent planes to the $\Sigma_n$. 

First assume that the extrinsic curvatures of the $\Sigma_n$ are all equal to $H^2$. Then it follows from Propositions \ref{prop-phi1} and \ref{prop-phi2} that lifts of $\hat{\Sigma}_n$ to the universal cover smoothly converge to a lift of $\hat{\Sigma}$ to the universal cover on compact sets, which establishes the claim.  

In the general case that the $H_n$ are allowed to vary, we let $\Sigma_n(H^2)$ be the unique $k$-surface in the homotopy class of $\Sigma_n$ for $k=H^2$.  Choose lifts $\tilde{\Sigma}_n(H^2)$ and $\tilde{\Sigma}_n$ of $\Sigma_n(H^2)$ and $\Sigma_n$ to $\tilde{N}$, so that $\tilde{\Sigma}_n(H^2)$ and $\tilde{\Sigma}_n$ are at finite Hausdorff distance from each other, and intersect a fixed compact fundamental domain independent of $n$.  Then if the Hausdorff distance between the tangential lifts of $\tilde{\Sigma}_n(H^2)$ and $\tilde{\Sigma}_n$ tends to zero as $n$ tends to infinity, we will have reduced establishing the claim to the case that all $\Sigma_n$ were $H^2$ surfaces and be able to conclude by the same argument as in the previous paragraph.  

Therefore suppose that this fails, and that for some $\epsilon>0$ it is the case that some fixed compact set $K$ of the unit tangent bundle contains unit tangent vectors normal to $\tilde{\Sigma}_n(H^2) \cap K $ at a distance of at least $\epsilon>0$ from  $\tilde{\Sigma}_n \cap K$.  Then after passing to subsequences the  $\tilde{\Sigma}_n(H^2)$ and $\tilde{\Sigma}_n$ converge on compact sets to k-surfaces $S_1$ and $S_2$ for $k=H^2$ by \cite{smith2024quaternions}[Theorem 1.1.1, Remark 1.1.2].  But $S_1$ and $S_2$ have the same asymptotic boundary, and so $S_1=S_2$ by the main results of \cite{labourie1999lemme} (see also \cite{smith2021asymptotic}), which is a contradiction.

\section{Asymptotically dense but not asymptotically  geodesic minimal surfaces} \label{section:notgeodesic}

In this section, we prove Theorem \ref{lastthmintro} from the introduction.  We construct a sequence of essential surfaces built from pants as follows. 

Recall the terminology from Section \ref{section:pants}. For $\eps>0$ and $R(\eps)>0$ sufficiently large, let $S^1_{\eps}$ and $S^2_{\eps}$ be $(1+O(\eps))$-quasifuchsian closed connected surfaces built out of at least one copy of each pants in $\pants$, all glued via $(\eps,R)$-good gluings.  For a fixed small $\theta$, we wish to build a surface $S_{\eps,\theta}$, by cutting $S^1_{\eps}$ and $S^2_{\eps}$ along a cuff $\gamma\in\curves$ and regluing them at an angle approximately equal to $\theta$.

Precisely, fix a cuff $\gamma\in\curves$. By the equidistribution of feet, there are pants $\pi_1^-\in \pants^-(\gamma)$ and $\pi_2^-\in \pants^-(\gamma)$, lying respectively on $S^1$ and $S^2$ so that
\[
|\ft \pi_1^- - \ft\pi_2^- - i\theta \,| <\frac{\eps}{R}.
\]
We also know there are $\pi_1^+ \in \pants^+(\gamma)$ and $\pi_2^+ \in \pants^+(\gamma)$, lying respectively on $S^1$ and $S^2$ that satisfy
\[
|\ft \pi_1^- - \tau(\ft \pi_1^+) | < \frac{\eps}{R} \aand
|\ft \pi_2^- - \tau(\ft \pi_2^+) | < \frac{\eps}{R},
\]
where $\tau: \N^1 (\sqrt{\gamma}) \to \N^1(\sqrt{\gamma})$ is the map $\tau:x\to x + 1 + i \pi$. We obtain our surface $S_{\eps,\theta}$ by cutting both $S^1$ and $S^2$ along $\gamma$ and regluing them by gluing $\pi_1^-$ to $\pi_2^+$ and $\pi_1^+$ to $\pi_2^-$.  It follows from \cite{KMC} that $S_{\eps,\theta}$ is $K(\theta)$-quasifuchsian.

We form a good matching of $(\epsilon,R)$ good pants as before for $\epsilon \rightarrow 0$, $R \to \infty$, except that we choose one pair of pants to glue together at a small angle $\theta$ along their cuff, for $\theta$ independent of $n$.  Denote the sequence of surfaces thus obtained by $S_n(\theta)$.  Choose a minimal surface $\Sigma_n(\theta)$ in the homotopy class of $S_n(\theta)$.  

Choose a Dirichlet fundamental domain for the action of $\pi_1(M)$ on $\mathbf{H}^3$, and lift the closed geodesic at which pants in $S_n(\theta)$ are glued at a definite angle to a geodesic $\gamma_n(\theta)$ in $\mathbf{H}^3$ that intersects $\Delta$ and that is contained in a lift $\tilde{S}_n(\theta)$ of $S_n(\theta)$.  The \textit{width} of the convex hull of a quasicircle in $\partial_{\infty}(\mathbf{H}^3)$ is the supremal distance between points in different components of its boundary.  We know that the width of the convex core of the limit set of  $\tilde{S}_n(\theta)$ is at least some positive constant just depending on $\theta$.  This implies that the $K$ for which $S_n(\theta)$ is $K$-quasi-conformal is uniformly bounded away from 1 (see e.g. the proof of the main theorem of \cite{Se}, which implies a bound for the width of the convex hull in terms of the quasiconformal constant.)  Work by Epstein \cite{epstein1986hyperbolic}[Equation 5.5] then implies that the minimal surface $\Sigma_n(\theta)$ must have principal curvatures uniformly bounded away from 1 by a constant that depends only on $\theta$.  

On the other hand, choose a sequence of points $p_n(\theta)$ on $S_n(\theta)$ so that:
\begin{enumerate} 
\item The tangent planes $\Pi_n$ to $S_n(\theta)$ are well-defined and converge to a plane $\Pi \in Gr_2(M)$ tangent to a dense totally geodesic plane in $Gr_2(M)$.  

\item The distance on $S_n(\theta)$ between $p_n(\theta)$ and $\gamma_n(\theta)$ tends to infinity.  

\end{enumerate}

Choose lifts $\tilde{S}_n(\theta)$ of $S_n(\theta)$ to $\mathbf{H}^3$ containing lifts $\tilde{p}_n(\theta)$ of $p_n(\theta)$ that are contained in $\Delta$, and let $\tilde{\Sigma}_n(\theta)$ be the corresponding lifts of $\Sigma_n(\theta)$ to $\mathbf{H}^3$.  Then by the second property above $\tilde{S}_n(\theta)$ converges to a totally geodesic plane on compact subsets of $\mathbf{H}^3$.  Therefore the convex hull of $\tilde{S}_n(\theta)$ converges to a totally geodesic plane on compact subsets of $\mathbf{H}^3$.  Since $\tilde{\Sigma}_n(\theta)$ is a minimal surface, it is contained in the convex hull, and so it must also converge to a totally geodesic plane.  It follows that the closure of $\Sigma_n(\theta)$ in $Gr_2(M)$ contains the closure of the totally geodesic plane tangent to $\Pi$ in $Gr_2(M)$, and so its closure is all of $Gr_2(M)$.

\bibliographystyle{amsalpha} 
	\bibliography{bibliography}

\end{document}